\documentclass[12pt]{article}

\usepackage{amsmath,amsthm,amssymb}
\usepackage[colorlinks,
            linkcolor=red,
            anchorcolor=blue,
            citecolor=magenta
            ]{hyperref}
\usepackage{pdfsync}
\usepackage[numbers,sort&compress]{natbib}
\allowdisplaybreaks
\theoremstyle{definition}

\theoremstyle{plain}
\newtheorem{theorem}{Theorem}[section]

\newtheorem{remark}{Remark}[section]
\newtheorem{lemma}{Lemma}[section]

\numberwithin{equation}{section}

\newcommand{\vs}{\vspace}

\begin{document}

\title{Transverse dependence and asymptotic dimensional reduction of ground states for the logarithmic Schr\"odinger equation on a waveguide manifold\footnote{This work was partially supported by NNSFC (No. 12171493).}}

\author{Jun Wang$^{a}$\footnote {Corresponding author. wangj937@mail2.sysu.edu.cn (J. Wang)
}, Hichem Hajaiej$^{b}$,  Zhaoyang Yin$^{a, c}$  \\
{\small $^{a}$Department of Mathematics, Sun Yat-sen University, Guangzhou, 510275, China } \\
{\small $^{b}$ Department of Mathematics, California State University LA, Los Angeles, CA 90032, USA}\\
{\small $^{c}$School of Science, Shenzhen Campus of Sun Yat-sen University, Shenzhen, 518107, China } \\
}

	\date{}

	\maketitle

\vs{-.7cm}

  \begin{abstract}
In this paper, we develop a novel approach to studying the $y$-dependence of ground states for the logarithmic Schr\"odinger equation on $\mathbb{R}^d \times \mathbb{T}$. Unlike in \cite{STNT2014} (Analysis \& PDE, 2014) and \cite{HH1} (SIAM J. Math. Anal., 2024), the power-law scaling argument does not transfer directly to the logarithmic setting. Instead, amplitude scaling gives an exact identity showing that transverse dependence is independent of the prescribed mass. We also discuss uniqueness for the Cauchy problem and conditional orbital stability.
\end{abstract}

{\footnotesize {\bf   Keywords:}  Logarithmic Schr\"odinger equation; Normalized solutions;  Variational methods; Waveguide manifold.

}

\section{ Introduction and main results}
This paper studies some properties of the solutions of the following logarithmic Schr\"odinger equation
\begin{equation} \label{eq1.1}
 \left\{\aligned
&-\Delta_{x,y} u+\lambda u =u\log u^2 ,\ (x, y) \in \mathbb{R}^d \times \mathbb{T}, \\
&\int_{\mathbb{R}^d \times \mathbb{T}}u^2dxdy=\Theta^2,
\endaligned
\right.
\end{equation}
where the mass $\Theta>0$, $\Delta_{x,y}=\sum_{j=1}^d\partial_{x_j}^2+\partial_y^2$, $\Delta_y$ is the Laplace--Beltrami operator on $\mathbb{T}$, and the frequency $\lambda$ is unknown and must be determined. In recent years, the logarithmic Schr\"odinger equation has received considerable attention. This class of
equations has some important physical applications in quantum mechanics, quantum optics, nuclear
physics, transport and diffusion phenomena, open quantum systems, effective quantum gravity, theory of
superfluidity, and Bose--Einstein condensation; see \cite{KGZ2010} and the references therein. On the whole space, the logarithmic Schr\"{o}dinger equation has received considerable attention because of its applications in various fields; see \cite{Białynicki-Birula} for a more detailed account. The main challenges arise from the fact that the nonlinearity in the energy functional is subcritical and that the subspace of radial functions in $H^1(\mathbb{R}^d)$ does not embed compactly into $L^2(\mathbb{R}^d)$. Therefore, classical variational approaches cannot directly establish the compactness of a minimizing sequence. Additionally, the functional is not smooth in $H^1(\mathbb{R}^d)$. To the best of our knowledge, the first rigorous mathematical study of \eqref{eq1.1} on the whole space was undertaken in \cite{Cazenave2}. This opened the door to many other contributions, such as \cite{Avenia-Montefusco-Squassina}. More recently, new developments have addressed more general sublinear nonlinearities; see \cite{Mederski-Schino} and \cite{Gallo-Schino}. Systems with critical exponents were studied in \cite{Hajaiej-Liu-Song-Zou} and \cite{Hajaiej-Tianhao-Liu-Wenming Zou}. Various aspects of the Cauchy problem have been addressed by Carles et al. in several papers; see, for example, \cite{RCIG2018}.

Recently, there has been increasing interest in studying dispersive equations on the waveguide manifolds $\mathbb{R}^d\times\mathbb{T}$, whose mixed geometric structure makes the underlying analysis rather challenging and delicate. For the nonlinear Schr\"odinger equation with a pure-power nonlinearity, there have been many results on waveguide manifolds, as shown in \cite{{XCZG2020},{ZHBP2014},{RKJM2021},{YL2023},{ZZH2021},{ZZJZ2021},{NTNV2012},{ADI2012},{XYHY2024}}. However, we are not aware of any previous results for the logarithmic Schr\"odinger equation on waveguide manifolds. Terracini et al. in \cite{STNT2014} considered the Cauchy problem
\begin{equation*}
  i\partial_tu-\Delta_{x,y} u-u|u|^\alpha=0,\ u|_{t=0}=u_0,
\end{equation*}
where $(x,y)\in\mathbb{R}_x^d\times M_y^k$ and $M_y^k$ is a compact Riemannian manifold. They proved that, above a critical mass, the ground states have nontrivial $M_y^k$-dependence. Subsequently, Hajaiej et al. in \cite{HH1} considered the following equation:
\begin{equation*}
  \Delta_{x,y}^2u-\beta \Delta_{x,y}u + \theta u = |u|^{\alpha}u, \ (x,y)\in\mathbb{R}^d\times\mathbb{T}^n ,
\end{equation*}
where $\beta \in \mathbb{R}$ and $\alpha \in (0, \frac{8}{d+n})$. They extended the conclusion in \cite{STNT2014} to a fourth-order Schr\"odinger equation by using a new scaling argument. It is natural to ask whether the same conclusion holds for Schr\"odinger equations without scaling structures, such as the Schr\"odinger--Poisson equation, the Schr\"odinger equation with a potential, the logarithmic Schr\"odinger equation, and the Choquard equation. This paper addresses the logarithmic equation by developing a new approach to studying the dependence of ground-state solutions.

The auxiliary functional $\mathcal I_\mu$ is motivated by \cite{STNT2014}.
A product test function proves transverse dependence for small $\mu$.
For large $\mu$, a slice-mass estimate gives convergence of the energy and
$\mu\|\partial_yu_\mu\|_2^2\to0$. The limit equation is obtained by averaging
over the periodic variable before passing to the limit. These asymptotic
statements do not establish exact independence at a finite parameter.
The natural energy functional associated with the equation \eqref{eq1.1} is given by
\begin{equation*}
\mathcal{I}(u)=\frac{1}{2} \int_{\mathbb{R}^d\times\mathbb{T}}(|\nabla_{x,y}u|^2+|u|^2) d xdy -\frac{1}{2} \int_{\mathbb{R}^d\times\mathbb{T}}u^{2}\log u^2 d xdy.
\end{equation*}
However, $\mathcal{I}(u)$ is not well defined in $H^1(\mathbb{R}^d\times\mathbb{T})$ since there exist functions $u\in H^1(\mathbb{R}^d\times\mathbb{T})$ such that $\int_{\mathbb{R}^d\times\mathbb{T}}u^{2}\log u^2 d xdy=-\infty$, which gives the possibility that $\mathcal{I}(u)=\infty$. In order to handle this issue, we consider a decomposition of the form
$$
F_2(s)-F_1(s)=\frac{1}{2} s^2 \log s^2, \  \forall s \in \mathbb{R},
$$
where $F_1\in C^1$ and $F_1$ is a nonnegative convex function, $F_2\in C^1$ satisfies Sobolev subcritical growth. This decomposition has been explored in many works (see \cite{{MSAS2015},{COACJ2020},{COACJ22023}}). More precisely, fix $0<\delta<e^{-3}$ and let
\begin{equation*}
  F_1(s):=\left\{\aligned
&0,  &s  =0, \\
&-\frac{1}{2}  s ^2 \log s^2, & 0<|s|  <\delta, \\
&-\frac{1}{2} s^2\left(\log \delta^2+3\right)+2 \delta |s|-\frac{\delta^2}{2}, & |s| \geq \delta
\endaligned
\right.
\end{equation*}
and
$$
F_2(s):= \begin{cases}0, & |s|<\delta, \\ \frac{1}{2} s^2 \log \left(\frac{s^2}{\delta^2}\right)+2 \delta |s|-\frac{3}{2} s^2-\frac{\delta^2}{2}, & |s| \geq \delta\end{cases}
$$
for every $s \in \mathbb{R}$. Hence,
\begin{equation}\label{eq1.2}
  F_2(s)-F_1(s)=\frac{1}{2} s^2 \log s^2, \quad \forall s \in \mathbb{R} .
\end{equation}
It is easy to see that $F_1$ and $F_2$ satisfy the following properties:
\begin{itemize}
\item[$(f_1)$] $F_1$ is an even function with $F_1^{\prime}(s) s \geq 0$ and $F_1 \geq 0$. Moreover, $F_1 \in C^1(\mathbb{R}, \mathbb{R})$  is convex for $\delta \approx 0^{+}$.

\item[$(f_2)$] $F_2 \in C^1(\mathbb{R}, \mathbb{R}) \cap C^2((\delta,+\infty), \mathbb{R})$ and there exists $C=C_p>0$ such that
$$
\left|F_2^{\prime}(s)\right| \leq C|s|^{p-1}, \quad \forall s \in \mathbb{R},\ p \in\left(2,2^*\right).
$$

\item[$(f_3)$] $s \mapsto \frac{F_2^{\prime}(s)}{s}$ is a nondecreasing function for $s>0$ and a strictly increasing function for $s>\delta$.

\item[$(f_4)$] $\lim\limits_{s \rightarrow \infty} \frac{F_2^{\prime}(s)}{s}=\infty$.
\end{itemize}
This decomposition guarantees that $\mathcal{I}(u)$ can be expressed as the sum of a $C^1$-functional and a convex, lower semi-continuous functional. In this way, the critical point theory of functionals developed in \cite{AS1986} can be used to find solutions
of \eqref{eq1.1}. Following \cite{COACJ2024}, we consider the reflexive and separable Orlicz space defined by
\begin{equation*}
  L^{F_1}(\mathbb{R}^d\times\mathbb{T})=\{u\in L^1_{loc}(\mathbb{R}^d\times\mathbb{T}):\int_{\mathbb{R}^d\times\mathbb{T}}F_1(|u|)dxdy<+\infty\}
\end{equation*}
The Orlicz norm is the Luxemburg norm
$\|u\|_{L^{F_1}}=\inf\{k>0:\int F_1(|u|/k)\le1\}$.
The explicit $F_1$ and its complement satisfy global $\Delta_2$, as used
in the proof of Lemma \ref{L2.5}.
We define $X=H^1(\mathbb{R}^d\times\mathbb{T})\cap L^{F_1}(\mathbb{R}^d\times\mathbb{T})$ with respect to the norm $\|u\|_{X}:=\|u\|_{H^1}+\|u\|_{L^{F_1}}$. Obviously, the embeddings $X\hookrightarrow H^1$ and $X\hookrightarrow L^{F_1}$ are continuous. Therefore, we can rewrite the functional $\mathcal{I}: X\rightarrow\mathbb{R}$ as
\begin{eqnarray*}
\mathcal{I}(u)=\frac{1}{2} \int_{\mathbb{R}^d\times\mathbb{T}}[|\nabla_{x,y}u|^2+|u|^2] d xdy+\int_{\mathbb{R}^d\times\mathbb{T}} F_1(u) d xdy-\int_{\mathbb{R}^d\times\mathbb{T}} F_2(u) d xdy
\end{eqnarray*}
and the mass constraint manifold is defined by
\begin{equation*}
S_{\Theta}=\left\{u \in X:\|u\|_2 =\Theta\right\}.
\end{equation*}

The main results of this paper concern the minimizers of the minimization problems
$$
m_{\Theta, \mu}=\inf\{\mathcal{I}_\mu:u\in S_\Theta \},
$$
where
\begin{equation*}
  \mathcal{I}_\mu(u)=\frac{1}{2} \int_{\mathbb{R}^d\times\mathbb{T}}[|\nabla_{x}u|^2+|u|^2] d xdy +\frac{\mu}{2} \int_{\mathbb{R}^d\times\mathbb{T}}|\nabla_yu|^2dxdy-\frac{1}{2} \int_{\mathbb{R}^d\times\mathbb{T}}u^{2}\log u^2 d xdy.
\end{equation*}
Here $\mu\in(0,+\infty)$.
By the argument of Theorem \ref{t1.1}, with a fixed positive transverse coefficient, there exists $u_\mu$ such that $\mathcal{I}_\mu(u_\mu)=m_{\Theta, \mu}$. In order to study the $y$-dependence of the ground states, we need some notation on the Euclidean space $\mathbb{R}_x^d$. Define
\begin{equation*}
\widetilde{S}_{\rho}=\left\{u \in X(\mathbb{R}_x^d):\|u\|_2 =\rho\right\}
\end{equation*}
and
$$
\widetilde{m}_{\rho}=\inf\{\widetilde{\mathcal{I}}(u):u\in \widetilde{S}_{\rho} \},
$$
where
\begin{equation*}
  \widetilde{\mathcal{I}}(u)=\frac{1}{2} \int_{\mathbb{R}^d}(|\nabla_{x}u|^2+|u|^2) d x -\frac{1}{2} \int_{\mathbb{R}^d}u^{2}\log u^2 d x.
\end{equation*}

\begin{theorem}\label{t1.1}
The problem \eqref{eq1.1}
has a solution pair $(u,\lambda)$, where $u$ is a positive minimizer which, after an $x$-translation, is
radially symmetric and radially nonincreasing in the $x$ variable.
\end{theorem}
The proof of Theorem \ref{t1.1} is based on the concentration--compactness principle. Next, we state the dependence result.
\begin{theorem}\label{t1.2}
For every $\Theta>0$, there exists $\mu_0(\Theta)>0$ such that every
minimizer of $m_{\Theta,\mu}$ has nontrivial $y$-dependence whenever
$0<\mu<\mu_0(\Theta)$. Moreover,
\[
\lim_{\mu\to\infty}m_{\Theta,\mu}
=2\pi\widetilde m_{\Theta/\sqrt{2\pi}},\qquad
\lim_{\mu\to\infty}\mu\|\partial_yu_\mu\|_2^2=0.
\]
The latter assertion is asymptotic and does not imply exact $y$-independence
at a finite value of $\mu$.
\end{theorem}
\begin{remark}\label{R1.1}
(1) Amplitude scaling gives $m_{\Theta,\mu}=\Theta^2m_{1,\mu}-\Theta^2\log\Theta$ and a bijection between the corresponding minimizer sets. Thus changing $\Theta$ cannot change their transverse dependence at fixed $\mu$.

(2) A nonzero finite-mass function cannot be independent of all Euclidean variables. For minimizers, \eqref{eq3.4} quantifies this fact.
\end{remark}
\begin{theorem}\label{t1.3}
Let $u_1,u_2\in C([-T,T];X)$ be distributional solutions of
\eqref{eq4.1}, with time derivatives in $L^\infty(-T,T;X')$.
Then
\[
 \|u_2(t)-u_1(t)\|_2^2
 \le e^{8|t|}\|u_2(0)-u_1(0)\|_2^2,\qquad |t|\le T.
\]
In particular, solutions in this class are unique. Existence of a global
conservative flow in $X$ is an additional hypothesis of Theorem \ref{t1.4},
not a conclusion of this theorem.
\end{theorem}
Theorem \ref{t1.3} is proved by an $L^2$ difference estimate. The following stability result is conditional on a global conservative flow.

\begin{theorem}\label{t1.4}
Define $\Gamma_\Theta=\{u\in S_\Theta, \mathcal{I}(u)=m_\Theta\}$. Assume that \eqref{eq4.1} admits a global continuous flow in the
complex energy space $X$ for initial data in a neighborhood $\mathcal U$
of $\Gamma_\Theta$ and that this flow conserves mass and $\mathcal I$.
Here $F_j(u)$ and $u^2\log u^2$ in the variational definitions are understood
as $F_j(|u|)$ and $|u|^2\log|u|^2$ for complex functions, and $\Gamma_\Theta$
denotes the full set of complex minimizers. Then the set $\Gamma_\Theta$ is orbitally stable in the sense that for all $\varepsilon > 0$ there exists
some $\delta = \delta (\varepsilon ) > 0$ such that, for any $\psi_0 \in \mathcal{U}$ satisfying
\begin{equation*}
  \inf_{u\in\Gamma_\Theta}\|\psi_0-u\|_{X}<\delta,
\end{equation*}
we have
\begin{equation*}
  \sup_{t\in\mathbb{R}}\inf_{u\in\Gamma_\Theta}\|\psi(t)-u\|_{X}<\varepsilon,
\end{equation*}
where $\psi$ is the global solution of \eqref{eq4.1} with $\psi (0) = \psi_0$.
\end{theorem}
\begin{remark}\label{R1.2}
Radial symmetry in $x$ alone does not imply compactness into $L^2$.
Indeed, a fixed-mass radial sequence may spread to larger Euclidean scales.
The logarithmic term, rather than radiality alone, is used below to exclude
vanishing.
\end{remark}

The paper is organized as follows. In Section 2, we study the existence of normalized solutions. Next, we study the $y$-dependence of the ground states. Finally, we consider the Cauchy problem and the orbital stability of the solution obtained in Theorem \ref{t1.1}.

\section{Existence of normalized solutions}
In this section, we study the existence of normalized solutions. To obtain our results, we need the Gagliardo--Nirenberg inequality on $\mathbb{R}^d \times \mathbb{T}$; see \cite{STNT2014}.

\begin{lemma}\label{L2.1}(Gagliardo-Nirenberg inequality). On $\mathbb{R}^d \times \mathbb{T}$ we have the Gagliardo-Nirenberg inequality
$$
\|u\|_{L^{2+\alpha}}^{2+\alpha} \leq C_{d,n,\alpha}\|  u\|_{H^1(\mathbb{R}^d \times \mathbb{T})}^{\vartheta(\alpha)} \|u\|_{L^2(\mathbb{R}^d \times \mathbb{T})}^{2+\alpha-\vartheta(\alpha)},
$$
where $\vartheta(\alpha)=\frac{(d+1)\alpha}{2}$.
\end{lemma}
\begin{lemma}\label{L2.2}
The functional $\mathcal{I}$ is coercive and bounded from below on $S_\Theta$.
\end{lemma}
\begin{proof}
By $(f_2)$, for each fixed $p\in\left(2,2+\frac{4}{d+1}\right)$ there exists
$C_p>0$ such that
$$
\left|F_2^{\prime}(u)\right| \leq C_p|u|^{p-1},\ \forall u \in \mathbb{R}.
$$
The Gagliardo--Nirenberg inequality and $F_1\ge0$ give
\begin{eqnarray*}
\mathcal{I}(u)&=&\frac{1}{2} \int_{\mathbb{R}^d\times\mathbb{T}}[|\nabla_{x,y}u|^2+|u|^2] d xdy+\int_{\mathbb{R}^d\times\mathbb{T}} F_1(u) d xdy-\int_{\mathbb{R}^d\times\mathbb{T}} F_2(u) d xdy \\
& \geq& \frac12\|\nabla_{x,y}u\|_2^2+\frac12\Theta^2
-C_{p,\Theta}\bigl(1+\|\nabla_{x,y}u\|_2^{\frac{(d+1)(p-2)}2}\bigr).
\end{eqnarray*}
Since $(d+1)(p-2)/2<2$, Young's inequality bounds the gradient on every energy sublevel.
Then $\int F_2(u)$ is bounded by the subcritical estimate, and the energy
identity bounds $\int F_1(u)$. The Luxemburg norm is bounded by
$\max\{1,\int F_1(u)\}$, so the whole $X$ norm is bounded.
For any fixed $\mu>0$ the same proof applies to $\mathcal I_\mu$.
\end{proof}

By using Lemma \ref{L2.2}, we can define $m_{\Theta}=\inf\limits_{u \in S_{\Theta}} \mathcal{I}(u)$. Next, we explore some properties of $m_{\Theta}$.

\begin{lemma}\label{L2.3}
$m_{a\Theta}=a^2m_\Theta-a^2\Theta^2\log a$ for every $a>0$.
\end{lemma}
\begin{proof}
The map $u\mapsto au$ is a bijection from $S_\Theta$ onto $S_{a\Theta}$ and
$\mathcal I(au)=a^2\mathcal I(u)-a^2\Theta^2\log a$. Taking infima proves
the identity.
\end{proof}
\begin{lemma}\label{L2.4}
For $0<\Theta_1<\Theta_2, \frac{\Theta_1^2}{\Theta_2^2} m_{\Theta_2}<m_{\Theta_1}$ holds.
\end{lemma}
\begin{proof}
Let $\xi>1$ be such that $\Theta_2=\xi \Theta_1$. Note that $\mathcal{I}(u)=\mathcal{I}(|u|)$ for all $u \in X$, so we may choose a nonnegative minimizing sequence $\left(u_n\right) \subset S_{\Theta_1}$ for $m_{\Theta_1}$; that is, $\mathcal{I} (u_n) \rightarrow m_{\Theta_1}$ as $n \rightarrow+\infty$. Setting $v_n=\xi u_n$, we clearly have $v_n \in S_{\Theta_2}$. Thus,
$$
m_{\Theta_2} \leq \mathcal{I}\left(v_n\right)=\xi^2 \mathcal{I} \left(u_n\right)- \xi^2 \log \xi  \int_{\mathbb{R}^d\times\mathbb{T}}\left|u_n\right|^2 d xdy=\xi^2 \mathcal{I} \left(u_n\right)- \Theta_1^2 \xi^2 \log \xi.
$$
Letting $n \rightarrow+\infty$, it follows from $\xi>1$ that
$$
m_{\Theta_2} \leq \xi^2 m_{\Theta_1}- \Theta_1^2 \xi^2 \log \xi  <\xi^2 m_{\Theta_1},
$$
which proves the lemma.
\end{proof}
As in the proof of Lemma 3.1 in \cite{WS2019}, we have the following logarithmic-type Br$\mathrm{\acute{e}}$zis--Lieb lemma on a waveguide manifold.
\begin{lemma}\label{L2.5}
Let $\left(u_n\right)$ be a bounded sequence in $X$ such that $u_n \rightarrow u$ a.e. in $\mathbb{R}^d\times\mathbb{T}$ and $\left\{u_n^2 \log u_n^2\right\}$ is a bounded sequence in $L^1\left(\mathbb{R}^d\times\mathbb{T}\right)$, then
$$
\lim _{n \rightarrow \infty} \int_{\mathbb{R}^d\times\mathbb{T}}\left(u_n^2 \log u_n^2-\left|u_n-u\right|^2 \log \left|u_n-u\right|^2\right) dxdy=\int_{\mathbb{R}^d\times\mathbb{T}} u^2 \log u^2 d xdy.
$$
\end{lemma}
\begin{proof}
For the convex function $A=F_1$ (with $\delta$ fixed sufficiently small),
direct differentiation of its two pieces gives constants
$1<\ell\le L<\infty$ with $\ell A(s)\le sA'(s)\le LA(s)$ for $s>0$.
In particular $A$ and its complementary function satisfy global $\Delta_2$.
Convexity, Young's inequality for complementary functions, and $\Delta_2$
give, for each $\varepsilon>0$,
\[
 |A(|a+b|)-A(|a|)|\le\varepsilon A(|a|)+C_\varepsilon A(|b|).
\]
For $F_2$ the subcritical derivative bound gives instead
\[
 |F_2(|a+b|)-F_2(|a|)|
 \le\varepsilon |a|^p+C_{\varepsilon,p}|b|^p
\]
for any fixed $2<p<2^*$. Apply these inequalities with $a=u_n-u$ and $b=u$.
After subtracting $F_j(|u|)$, the positive part of the absolute error minus
the indicated $\varepsilon$ term tends pointwise to zero and is dominated
by a fixed integrable function of $u$. Dominated convergence, boundedness
of $\int A(|u_n-u|)$ and $\|u_n-u\|_p^p$, and then $\varepsilon\downarrow0$
prove splitting for $F_1$ and $F_2$ separately. Subtracting the identities
and using \eqref{eq1.2} proves the logarithmic splitting. This argument also
applies to complex functions.
\end{proof}
Next, we prove a compactness lemma on $S_{\Theta}$.
\begin{lemma}\label{L2.6} Let $\left(u_n\right) \subset S_\Theta$ be a minimizing sequence with respect to $m_{\Theta}$. Then, for some subsequence either

$\mathrm{(i)}$ $(u_n)$ is strongly convergent in $X$, or

$\mathrm{(ii)}$ there exists $\left(z_n\right) \subset \mathbb{R}^d$ with
$\left|z_n\right| \rightarrow+\infty$ such that
$v_n(x,y)=u_n\left(x+z_n,y\right)$ converges strongly in $X$ to a function
$v\in S_\Theta$ satisfying $\mathcal{I}(v)=m_\Theta$.
\end{lemma}
\begin{proof}
According to Lemma \ref{L2.2}, the sequence $\left(u_n\right)$ is bounded in $X$ and $u_n \rightharpoonup u$ in $X$ for some subsequence. If $u \neq 0$ and $\|u\|_2=b \neq \Theta$, we must have $b \in(0, \Theta)$. By Br$\mathrm{\acute{e}}$zis-Lieb lemma, it holds $\|u_n\|_2^2=\|u_n-u\|_2^2+\|u\|_2^2+o_n(1)$.
Moreover, it follows from Lemma \ref{L2.5} that
$$
\int_{\mathbb{R}^d\times\mathbb{T}} u_n^2 \log u_n^2 d xdy=\int_{\mathbb{R}^d\times\mathbb{T}}\left|u_n-u\right|^2 \log \left|u_n-u\right|^2 d xdy+\int_{\mathbb{R}^d\times\mathbb{T}} u^2 \log u^2 d xdy+o_n(1).
$$
Let $v_n=u_n-u$ and $d_n=\|v_n\|_2$. Supposing that $d_n \rightarrow d$, we have $\Theta^2=b^2+d^2$ and $d_n \in(0, \Theta)$ for all sufficiently large $n \in \mathbb{N}$. Hence,
$$
m_{\Theta}+o_n(1)=\mathcal{I}\left(u_n\right)=\mathcal{I} \left(v_n\right)+\mathcal{I}(u)+o_n(1) \geq m_{d_n}+m_{b}+o_n(1).
$$
By using Lemma \ref{L2.4}, we have $m_{\Theta}+o_n(1) \geq \frac{d_n^2}{\Theta^2} m_{\Theta}+m_{b}+o_n(1)$. Letting $n \rightarrow+\infty$, it holds $m_{\Theta} \geq \frac{d^2}{\Theta^2} m_{\Theta}+m_{b}$. Note that $b \in(0, \Theta)$, employing again Lemma \ref{L2.4}, we obtain that
$$
m_{\Theta}>\frac{d^2}{\Theta^2} m_{\Theta}+\frac{b^2}{\Theta^2} m_{\Theta}=\left(\frac{d^2}{\Theta^2}+\frac{b^2}{\Theta^2}\right) m_{\Theta}=m_{\Theta},
$$
which is absurd, so $\|u\|_2=\Theta$. Since $\|u_n\|_2= \|u\|_2=\Theta$, $u_n \rightharpoonup u$ in $L^2(\mathbb{R}^d\times\mathbb{T})$, and $L^2(\mathbb{R}^d\times\mathbb{T})$ has the weak-convergence-plus-norm-convergence property, we obtain $u_n \rightarrow u $ in $L^2(\mathbb{R}^d\times\mathbb{T})$. Using the interpolation theorem in Lebesgue spaces and $(f_2)$, one has
$$
\int_{\mathbb{R}^d\times\mathbb{T}} F_2\left(u_n\right) d xdy \rightarrow \int_{\mathbb{R}^d\times\mathbb{T}} F_2(u) d xdy.
$$
These limits together with the limit $m_{\Theta}=\lim\limits_{n \rightarrow+\infty} \mathcal{I} (u_n)$ and the fact that $F_1 \geq 0$ provide $m_{\Theta} \geq \mathcal{I}(u)$.
Since $u \in S_\Theta$, we have $\mathcal{I}(u)=m_{\Theta}$. Therefore, $\lim\limits_{n \rightarrow+\infty} \mathcal{I}(u_n)=\mathcal{I}(u)$ and
$$
\int_{\mathbb{R}^d\times\mathbb{T}}\left|\nabla u_n\right|^2 d xdy \rightarrow \int_{\mathbb{R}^d\times\mathbb{T}}|\nabla u|^2 d xdy,\ \int_{\mathbb{R}^d\times\mathbb{T}} F_1\left(u_n\right) d xdy \rightarrow \int_{\mathbb{R}^d\times\mathbb{T}} F_1(u) d xdy.
$$
To justify the last upgrade, the nonnegative functions $F_1(|u_n|)$
converge almost everywhere and their integrals converge. Scheffe's lemma gives
convergence in $L^1$. The global $\Delta_2$ bound for this $F_1$ and convexity give
\[
 F_1(|u_n-u|)\le C\{F_1(|u_n|)+F_1(|u|)\}.
\]
The right-hand side is uniformly integrable and has uniformly small tails.
By Vitali's theorem, $\int F_1(|u_n-u|)\to0$. In particular, the
logarithmic part satisfies the Brezis--Lieb decomposition
\[
\int |u_n|^2\log |u_n|^2
=\int |u_n-u|^2\log |u_n-u|^2+\int |u|^2\log |u|^2+o(1),
\]
which follows by applying the splitting separately to the convex Orlicz part
$F_1$ and the subcritical part $F_2$.
Iterating $\Delta_2$, for every fixed $\varepsilon>0$,
$\int F_1(|u_n-u|/\varepsilon)\to0$, which is precisely Luxemburg-norm
convergence. Hence $u_n\to u$ in $X$.

Next suppose that the original weak limit is zero. If no translated
subsequence has nonzero local mass, the waveguide vanishing lemma gives
$\|u_n\|_p\to0$ for any fixed $2<p<2^*$. Jensen's inequality for the
probability measure $|u_n|^2\,dxdy/\Theta^2$ yields
\begin{equation}\label{eq2.1}
 \int |u_n|^2\log|u_n|^2
 \le \frac{2\Theta^2}{p-2}
       \log\!\left(\frac{\|u_n\|_p^p}{\Theta^2}\right)
 \longrightarrow-\infty.
\end{equation}
(The logarithm may first be truncated to justify Jensen's inequality.)
Since the kinetic energy is nonnegative, this forces
$\mathcal I(u_n)\to+\infty$, a contradiction. Thus, after a subsequence,
there exist $R,\beta>0$ and $z_n\in\mathbb R^d$ such that
\begin{equation}\label{eq2.2}
 \int_{B_R(z_n)\times\mathbb T}|u_n|^2\,dxdy\ge\beta.
\end{equation}
Local Rellich compactness gives a nonzero weak limit for
$u_n(x+z_n,y)$. Applying the first part of the proof gives strong convergence
in $X$. If the original weak limit is zero, local compactness also gives
$|z_n|\to\infty$ after extraction.
\end{proof}

\noindent\textbf{Proof of Theorem \ref{t1.1}}
By Lemma \ref{L2.2}, there exists a bounded minimizing sequence $\left(u_n\right) \subset S_\Theta$ such that $\mathcal{I}(u_n) \rightarrow m_{\Theta}$. By Lemma \ref{L2.6}, there exists $u \in S_\Theta$ with $\mathcal{I}(u)=m_{\Theta}$. Therefore, by the Lagrange multiplier theorem, there exists $\lambda_\Theta \in \mathbb{R}$ such that $\mathcal{I}^{\prime}(u)=-\lambda_\Theta \Psi^{\prime}(u)$ in $X^{\prime}$, where $\Psi: X \rightarrow \mathbb{R}$ is defined by $\Psi(u)=\frac{1}{2} \int_{\mathbb{R}^d\times\mathbb{T}} u^2 d xdy,\ u \in X$. Hence, we have
$$
-\Delta_{x,y} u +\lambda_\Theta u=u \log u^2  \  \text { in } \mathbb{R}^d\times\mathbb{T}.
$$
Testing the equation with $u$ gives
$$
\int_{\mathbb{R}^d\times\mathbb{T}} |\nabla_{x,y} u|^2 d xdy+\lambda_\Theta \Theta^2=\int_{\mathbb{R}^d\times\mathbb{T}} u^2 \log u^2 d xdy.
$$
and hence
$$
m_{\Theta}=\mathcal{I}(u) = \frac{(1-\lambda_\Theta) \Theta^2}{2}.
$$
Thus $\lambda_\Theta=1-2m_\Theta/\Theta^2$; no mass-independent sign
condition on $\lambda_\Theta-1$ is asserted.
Next, we claim that $u$ can be chosen as a positive function. In fact, it is easy to check that $\mathcal{I}(|u|)=\mathcal{I}(u)$. Moreover, $|u| \in S_\Theta$ because of $u \in S_\Theta$, so
$m_{\Theta}=\mathcal{I}(u)=\mathcal{I}(|u|) \geq m_{\Theta}$, which shows that $\mathcal{I}(|u|)=m_{\Theta}$, so  we can replace $u$ by $|u|$. Furthermore, if $u^*$ denotes Schwarz symmetrization of each $x$-slice of $|u|$ (leaving $y$ periodic), we have
$$
\int_{\mathbb{R}^d\times\mathbb{T}}|\nabla u|^2 d xdy \geq \int_{\mathbb{R}^d\times\mathbb{T}}\left|\nabla u^*\right|^2 d xdy, \ \int_{\mathbb{R}^d\times\mathbb{T}}|u|^2 d xdy=\int_{\mathbb{R}^d\times\mathbb{T}}\left|u^*\right|^2 d xdy
$$
and by \cite{HH2}
$$
\int_{\mathbb{R}^d\times\mathbb{T}} F_1(u) d xdy=\int_{\mathbb{R}^d\times\mathbb{T}} F_1\left(u^*\right) d xdy, \ \int_{\mathbb{R}^d\times\mathbb{T}} F_2(u) d xdy=\int_{\mathbb{R}^d\times\mathbb{T}} F_2\left(u^*\right) d xdy,
$$
then $u^* \in S_\Theta$ and $\mathcal I\left(u^*\right)=m_{\Theta}$, from which it follows that we can replace $u$ by $u^*$. Thus, we may choose a minimizer that is radially symmetric modulo translations.
Moreover, the logarithmic elliptic strong maximum principle applies because
$|s\log s^2|\le C_\varepsilon(|s|^{1-\varepsilon}+|s|^{1+\varepsilon})$ near the origin.
Therefore a nonnegative nontrivial minimizer is strictly positive in
$\mathbb{R}^d\times\mathbb{T}$.

\section{Dependence of the ground states}

\subsection{Parameter tends to zero}
Now, we consider the case when $\mu\rightarrow0$.
\begin{lemma}\label{L3.1}
$\lim\limits_{\mu\rightarrow0}m_{\Theta,\mu}<2\pi\widetilde{m}_{\frac{\Theta}{\sqrt{2\pi}}}$.
\end{lemma}
\begin{proof}
Let $v$ minimize $\widetilde m_{\Theta/\sqrt{2\pi}}$, set
$Q=\sqrt{2\pi}v$, and choose nonconstant $h\ge0$ in $C^\infty(\mathbb T)$
with $\|h\|_2=1$. Direct expansion gives
\[
 \mathcal I_\mu(Qh)-2\pi\widetilde m_{\Theta/\sqrt{2\pi}}
 =\frac{\Theta^2}{2}\left(\mu\|h'\|_2^2
 -\int_{\mathbb T}h^2\log(2\pi h^2)\,dy\right).
\]
The integral is strictly positive by Jensen's inequality. Hence the
right-hand side is negative for all sufficiently small $\mu>0$.
\end{proof}
\subsection{Parameter tends to infinity}
Now, we consider the case when $\mu\rightarrow+\infty$.
\begin{lemma}\label{L3.2}
Let $u_\mu$ be an optimizer of $m_{\Theta, \mu}$. Then
\begin{equation}\label{eq3.2}
  \lim\limits_{\mu \rightarrow \infty}m_{\Theta,\mu}
  =2\pi\widetilde m_{\Theta/\sqrt{2\pi}}
\end{equation}
and
\begin{equation}\label{eq2.3}
  \lim\limits_{\mu \rightarrow \infty}  \mu\left\|u_{y,\mu}\right\|_2^2=0.
\end{equation}
\end{lemma}
\begin{proof}
A $y$-independent Euclidean optimizer gives the corresponding upper bound.
Since $\mathcal I_1$ is bounded below on $S_\Theta$,
\[
m_{\Theta,\mu}=\mathcal I_1(u_\mu)
+\frac{\mu-1}{2}\|\partial_yu_\mu\|_2^2
\]
first implies $\|\partial_yu_\mu\|_2=O(\mu^{-1/2})$. Set
$\rho_\mu(y)=\|u_\mu(\cdot,y)\|_{L_x^2}^2$. Then
$\int_{\mathbb T}\rho_\mu=\Theta^2$ and
$\|\rho_\mu'\|_1\le2\Theta\|\partial_yu_\mu\|_2\to0$, whence
$\rho_\mu\to\Theta^2/(2\pi)$ uniformly. The slice variational inequality is
\[
m_{\Theta,\mu}\ge\frac\mu2\|\partial_yu_\mu\|_2^2
+\int_{\mathbb T}\widetilde m_{\sqrt{\rho_\mu(y)}}\,dy.
\]
Using the exact amplitude identity from Lemma \ref{L2.3}, the last integral
converges to $2\pi\widetilde m_{\Theta/\sqrt{2\pi}}$. Comparison with the
upper bound proves both conclusions.
\end{proof}
\subsection{Dependence analysis}
\begin{lemma}\label{L3.4}
We have
\begin{equation}\label{eq3.4}
   \int_{\mathbb{R}^d\times\mathbb{T}}|\nabla_{x}u_{\mu}|^2dxdy=\frac{d}{2}\Theta^2
\end{equation}
for $\mu>0$. Then there exists $\lambda_{\mu}\in\mathbb{R}$ such that
\begin{equation}\label{eq3.5}
   -\Delta_{x}u_\mu-\mu\Delta_{y}u_\mu+\lambda_{\mu}u_\mu=u_\mu\log u_\mu^2.
\end{equation}
Moreover,
\begin{equation}\label{eq3.6}
  \lim\limits_{\mu\rightarrow\infty}\lambda_{\mu}=\overline{\lambda}.
\end{equation}
\end{lemma}
\begin{proof}
Since $u_{\mu}$ is a constrained minimizer of $\mathcal{I}_\mu(u)$ on the sphere of radius $\Theta$ in $L^2(\mathbb{R}^d\times\mathbb{T})$, it follows that $\frac{d}{dt}\mathcal{I}_\mu(u_{\mu}^t)\big|_{t=1}=0$, where $u_{\mu}^t=t^{\frac{d}{2}}u_{\mu}(tx,y)$. Moreover, we have
\begin{eqnarray*}
  \frac{d}{dt}\mathcal{I}_\mu(u_{\mu}^t)\big|_{t=1} &=&\frac{d}{dt}\left[\frac{t^2}{2} \int_{\mathbb{R}^d\times\mathbb{T}}|\nabla_{x}u_{\mu}|^2dxdy+\frac{\Theta^2}{2}+\frac{\mu}{2} \int_{\mathbb{R}^d\times\mathbb{T}}|\nabla_yu_{\mu}|^2dxdy\right.\\
  &&\left.-\frac{1}{2} \int_{\mathbb{R}^d\times\mathbb{T}}u_{\mu}^{2}\log u_{\mu}^2 d xdy-\frac{d}{2} \int_{\mathbb{R}^d\times\mathbb{T}}u_{\mu}^{2}\log t d xdy\right]\Big|_{t=1} \\
    &=&\int_{\mathbb{R}^d\times\mathbb{T}}|\nabla_{x}u_{\mu}|^2dxdy-\frac{d}{2}\Theta^2 ,
\end{eqnarray*}
which proves \eqref{eq3.4}. Equation \eqref{eq3.5} follows from the Lagrange multiplier technique. Now, using \eqref{eq3.5}, one has
\begin{equation*}
   \int_{\mathbb{R}^d\times\mathbb{T}}(|\nabla_{x}u_{\mu}|^2+\mu|\nabla_yu_{\mu}|^2) d xdy +\lambda_{\mu}\|u_{\mu}\|_2^2= \int_{\mathbb{R}^d\times\mathbb{T}}u_{\mu}^{2}\log u_{\mu}^2 d xdy.
\end{equation*}
Hence,
\begin{eqnarray*}
\lambda_{\mu}&=&\frac{1}{\Theta^2}\left[-\int_{\mathbb{R}^d\times\mathbb{T}}(|\nabla_{x}u_{\mu}|^2+\mu|\nabla_yu_{\mu}|^2) d xdy+\int_{\mathbb{R}^d\times\mathbb{T}}u_{\mu}^{2}\log u_{\mu}^2 d xdy\right] \\
&=&-\frac{2}{\Theta^2}\left[\frac{1}{2} \int_{\mathbb{R}^d\times\mathbb{T}}(|\nabla_{x}u_{\mu}|^2+|u_{\mu}|^2+\mu|\nabla_yu_{\mu}|^2) d xdy -\frac{1}{2} \int_{\mathbb{R}^d\times\mathbb{T}}u_{\mu}^{2}\log u_{\mu}^2 d xdy\right]+1\\
    &=&-\frac{2}{\Theta^2}m_{\Theta, \mu}+1,
\end{eqnarray*}
which proves \eqref{eq3.6}.
\end{proof}
\begin{lemma}\label{L3.5}
Let $u_\mu$ be nonnegative minimizers. There exists $v\in\widetilde S_{\Theta/\sqrt{2\pi}}$ such that, up to a
subsequence and $\mathbb R^d$-translations, $u_\mu$ converges strongly in
$X(\mathbb R^d\times\mathbb T)$ to the $y$-independent function $v$ as
$\mu\to\infty$, $\widetilde{\mathcal I}(v)
=\widetilde m_{\Theta/\sqrt{2\pi}}$, and
$v$ satisfies
\begin{equation}\label{eq3.7}
  -\Delta_xv+\overline{\lambda}v=v\log v^2.
\end{equation}
\end{lemma}
\begin{proof}
We choose the nonnegative minimizers throughout this lemma. Set
$M(a)=2\pi\widetilde m_{a/\sqrt{2\pi}}$ and $M(0)=0$.
Lemma \ref{L3.2} implies
$\mathcal I_{\mu_j}(u_{\mu_j})\to M(\Theta)$ and
$\mu_j\|\partial_yu_{\mu_j}\|_2^2\to0$.
The energy sublevel estimate gives boundedness in $X$.
The Jensen argument \eqref{eq2.1} excludes vanishing, since
$\mathcal I_1(u_{\mu_j})\le\mathcal I_{\mu_j}(u_{\mu_j})$ for $\mu_j\ge1$.
After translation and extraction, write $U_j(x,y)=u_{\mu_j}(x+\tau_j,y)$.
Then $U_j\rightharpoonup v\ne0$ in $X$, with local strong $L^2$ and
almost-everywhere convergence. The vanishing transverse derivative shows
$\partial_yv=0$, so $v=v(x)\ge0$.

We first justify the equation without taking a limit of
$\mu_j\partial_y^2U_j$. Let
$\overline U_j(x)=(2\pi)^{-1}\int_{\mathbb T}U_j(x,y)\,dy$.
Testing \eqref{eq3.5} with a function of $x$ alone gives exactly
\[
 -\Delta_x\overline U_j+\lambda_{\mu_j}\overline U_j
 =\frac1{2\pi}\int_{\mathbb T}U_j\log U_j^2\,dy.
\]
For small $\eta>0$, the bound
$|s\log s^2|\le C_\eta(|s|^{1-\eta}+|s|^{1+\eta})$
and local subcritical compactness imply convergence of the right side
in distributions. Therefore
\begin{equation}\label{eq3.8}
 -\Delta_xv+\overline\lambda v=v\log v^2
 \quad\hbox{in }\mathbb R^d.
\end{equation}
Here $\overline\lambda=1-2M(\Theta)/\Theta^2$, by the corrected identity
in Lemma \ref{L3.4}.

Next put $a=\sqrt{2\pi}\|v\|_{L_x^2}\in(0,\Theta]$,
$r_j=U_j-v$, and $b_j=\|r_j\|_2$. Hilbert-space splitting gives
$b_j^2\to b^2=\Theta^2-a^2$.
Because $\partial_yv=0$, the transverse kinetic energy of $U_j$ equals
that of $r_j$ exactly. The $x$-gradient splitting and Lemma \ref{L2.5}
therefore give
\[
 \mathcal I_{\mu_j}(U_j)
 =2\pi\widetilde{\mathcal I}(v)+\mathcal I_{\mu_j}(r_j)+o(1).
\]
To control the remainder, suppose $b>0$ and set
$q_j(y)=\|r_j(\cdot,y)\|_{L_x^2}^2$. We have
\[
 \int_{\mathbb T}q_jdy=b_j^2,\qquad
 \|q_j'\|_1\le2b_j\|\partial_yr_j\|_2\to0.
\]
The one-dimensional periodic Poincare estimate implies
$q_j\to b^2/(2\pi)$ uniformly. The slice inequality, which does not
require $r_j$ to be a minimizer, now yields
\[
 \mathcal I_{\mu_j}(r_j)
 \ge\frac{\mu_j}{2}\|\partial_yr_j\|_2^2
       +\int_{\mathbb T}\widetilde m_{\sqrt{q_j(y)}}\,dy.
\]
Continuity of $\rho\mapsto\widetilde m_\rho$ on $(0,\infty)$ gives
$\liminf_j\mathcal I_{\mu_j}(r_j)\ge M(b)$. Thus
$M(\Theta)\ge M(a)+M(b)$.
But the amplitude identity gives, for $a,b>0$ and $a^2+b^2=\Theta^2$,
\[
 M(a)+M(b)-M(\Theta)
 =\frac12\left[a^2\log\frac{\Theta^2}{a^2}
              +b^2\log\frac{\Theta^2}{b^2}\right]>0.
\]
This contradiction proves $a=\Theta$ and hence global strong $L^2$
convergence of $U_j$ to $v$.

Interpolation gives strong $L^p$ convergence for $2<p<2^*$ and hence
$\int F_2(U_j)\to\int F_2(v)$.
Furthermore
$\mathcal I_1(U_j)=\mathcal I_{\mu_j}(U_j)
-(\mu_j-1)\|\partial_yU_j\|_2^2/2\to M(\Theta)$.
Weak lower semicontinuity of the gradient and Fatou's lemma for $F_1$ give
\[
 M(\Theta)\le \mathcal I_1(v)\le\liminf_j\mathcal I_1(U_j)=M(\Theta).
\]
Subtracting the convergent $L^2$ and $F_2$ terms in the energy shows that
\[
 \frac12\|\nabla U_j\|_2^2+\int F_1(U_j)
 \longrightarrow\frac12\|\nabla v\|_2^2+\int F_1(v).
\]
Each summand is lower semicontinuous; convergence of the sum therefore
forces convergence of both summands. Hilbert-space norm convergence
gives strong $H^1$ convergence. The Scheffe--Vitali and $\Delta_2$ argument
in Lemma \ref{L2.6} gives strong $L^{F_1}$ convergence.
This proves the asserted $X$ convergence and the Euclidean minimizing
property. Equation \eqref{eq3.7} is \eqref{eq3.8}.
\end{proof}

\begin{remark}\label{R3.6}
The convergence $\mu\|\partial_yu_\mu\|_2^2\to0$ is only asymptotic. The
preceding estimates do not imply $\partial_yu_\mu=0$ for all sufficiently
large finite $\mu$, because $\log v$ is unbounded below as $|x|\to\infty$
and the estimate formerly used for $|\log v-\log u_\mu|$ is not valid without
a uniform positive lower bound. We therefore make no finite-parameter
rigidity claim.
\end{remark}

\noindent\textbf{Proof of Theorem \ref{t1.2}.} The small-$\mu$ assertion
follows from Lemma \ref{L3.1}. The large-$\mu$ assertions follow from the
slice estimate and the argument of Lemma \ref{L3.2}.

\section{The proof of Theorem \ref{t1.3}}
In this section, we consider the Cauchy problem associated with \eqref{eq1.1}, that is,
\begin{equation}\label{eq4.1}
  i\partial_tu+\Delta_{x,y} u+u\log |u|^2=0,\ u|_{t=0}=u_0.
\end{equation}
The variational existence theorem does not by itself give a global
conservative evolution. We prove the uniqueness assertion stated in
Theorem \ref{t1.3} and specify the remaining requirements for existence.

A possible regularization is
\begin{equation}\label{eq4.2}
 i\partial_tu^\varepsilon+\Delta_{x,y}u^\varepsilon
 +u^\varepsilon\log(\varepsilon+|u^\varepsilon|^2)=0.
\end{equation}
Its potential must be
\[
 H_\varepsilon(s)=\int_0^s\log(\varepsilon+r)\,dr
 =(\varepsilon+s)\log(\varepsilon+s)
   -\varepsilon\log\varepsilon-s.
\]
Thus the corresponding formal conserved energy is
$\frac12\|\nabla u^\varepsilon\|_2^2
-\frac12\int H_\varepsilon(|u^\varepsilon|^2)$.
Replacing the logarithm inside the original energy without changing its
primitive would not give the energy of \eqref{eq4.2}.

Although $X'=H^{-1}+(L^{F_1})'$ with its sum norm is the natural dual,
the regularized-energy bound alone does not establish a uniform
$L^\infty(-T,T;X)$ bound: near zero the regularized potential does not
control $F_1(s)$ uniformly as $s\to0$. A global construction would require
a separate Orlicz estimate, a uniform bound for the regularized nonlinear
term in $X'$, and a justified energy identity for the limit. Local
compactness alone supplies none of these equalities. References
\cite{Cazenave2,TC2003,RCIG2018} motivate this program, but no unverified
transfer of their hypotheses to this waveguide is used here.

\noindent\textbf{Proof of Theorem \ref{t1.3}.}
The pointwise inequality
\begin{equation}\label{eq4.3}
 \left|\operatorname{Im}
 \bigl[(z\log|z|^2-w\log|w|^2)(\bar z-\bar w)\bigr]\right|
 \le4|z-w|^2
\end{equation}
holds with the continuous convention at zero. Put $h=u_2-u_1$.
The equation, tested against $h$ in the real $X',X$ duality, gives
\[
 \left|\frac12\frac{d}{dt}\|h(t)\|_2^2\right|
 \le4\|h(t)\|_2^2.
\]
The kinetic pairing is real and contributes zero to this derivative.
Time regularization justifies this calculation under the stated
$X$ and $X'$ assumptions. Gronwall's inequality, forwards and backwards
from zero, proves the assertion.

For clarity, global mass conservation can also be justified for a
constructed solution in this class without inferring it from local
convergence. If $\chi_R(x)=\chi(x/R)$ is a real smooth cutoff, its
localized mass identity satisfies
\[
 \left|\frac{d}{dt}\int\chi_R|u|^2\right|
 \le \frac{C}{R}\|u\|_2\|\nabla_xu\|_2.
\]
Integration on a finite time interval followed by $R\to\infty$ proves
mass conservation. Establishing the solution class and energy conservation
is still required; these are explicit hypotheses below.

\section{Orbital stability}
\noindent\textbf{Proof of Theorem \ref{t1.4}.}
Suppose the conclusion fails. Then there are initial data $\psi_{n,0}$,
times $t_n$, and $\varepsilon_0>0$ with
\begin{equation}\label{eq5.1}
 \operatorname{dist}_X(\psi_{n,0},\Gamma_\Theta)\to0,
\end{equation}
but
\begin{equation}\label{eq5.2}
 \operatorname{dist}_X(\psi_n(t_n),\Gamma_\Theta)\ge\varepsilon_0.
\end{equation}
Choose $g_n\in\Gamma_\Theta$ with $\|\psi_{n,0}-g_n\|_X\to0$.
The energy sublevel bound gives uniform boundedness of $g_n$ and
$\psi_{n,0}$ in $X$. The derivative growth estimates for $F_1,F_2$ give
continuity of the energy uniformly on these bounded sets.
For $v_n=\psi_n(t_n)$, conservation therefore gives
\begin{equation}\label{eq5.3}
 \|v_n\|_2\to\Theta,\qquad \mathcal I(v_n)\to m_\Theta.
\end{equation}
The coercivity argument of Lemma \ref{L2.2}, uniformly for masses near
$\Theta$, bounds $v_n$ in $X$. Set $a_n=\Theta/\|v_n\|_2$ and
$\widetilde v_n=a_nv_n$. Then
\[
 \|\widetilde v_n-v_n\|_X
 =|a_n-1|\|v_n\|_X\to0,\qquad \|\widetilde v_n\|_2=\Theta.
\]
The exact amplitude identity gives
\begin{equation}\label{eq5.4}
 \mathcal I(\widetilde v_n)
 =a_n^2\mathcal I(v_n)-a_n^2\|v_n\|_2^2\log a_n
 \longrightarrow m_\Theta.
\end{equation}
The compactness proof of Lemma \ref{L2.6} applies to complex functions
using $F_j(|u|)$ and the complex Hilbert-space splitting.
After translation and extraction, $\widetilde v_n$ converges in $X$ to a
member of $\Gamma_\Theta$. This set contains all translations and constant
phases of its elements. Its distance from $v_n$ must therefore tend to zero,
contradicting \eqref{eq5.2}.
\hfill$\square$

\textbf{Data availability} This article has no additional data.

\textbf{Conflict of interest} On behalf of all authors, the corresponding author states that there is no conflict of
interest.

\end{document}